\documentclass[11pt]{article}

\usepackage[T1]{fontenc}
\usepackage{times}

\usepackage{mathptmx}
\usepackage{amsmath,amsthm,amssymb} 
\usepackage{imakeidx}
\makeindex

\newcommand{\mathbbm}{\mathbb} 

\usepackage{mathtools}
\usepackage{type1cm}

\usepackage{makeidx}         
\usepackage{graphicx}        
\usepackage{multicol}        
\usepackage[bottom]{footmisc}
\usepackage{tikz}
\usepackage{dsfont}

\usepackage{psfrag}

\usepackage{comment}

\usepackage{footnpag}

\newcommand{\Yc}{\mathcal{Y}}
\newcommand{\Fc}{\mathcal{F}}

\newcommand{\Gc}{\mathcal{G}}

\newcommand{\Lc}{\mathcal{L}}

\newcommand{\Nc}{\mathcal{N}}

\newcommand{\Sc}{\mathcal{S}}

\newcommand{\Dc}{\mathcal{D}}

\newcommand{\Xf}{\mathfrak{X}}

\newcommand{\Yf}{\mathfrak{Y}}

\newcommand{\Lf}{\mathfrak{L}}

\newcommand{\Df}{\mathfrak{D}}

\newcommand{\Rb}{{\mathbbm{R}}}

\newcommand{\Fb}{{\mathbbm{F}}}
\newcommand{\Eb}{{\mathbbm{E}}}

\newcommand{\Sb}{{\mathbbm{S}}}

\newcommand{\Nb}{{\mathbbm{N}}}

\newcommand{\Xb}{{\mathbbm{X}}}

\newcommand{\xb}{{\mathbbm{x}}}

\newcommand{\hb}{{\mathbbm{h}}}

\newcommand{\Cdot}{\hspace{0.1em}\cdot\hspace{0.1em}}
\newcommand{\eqdef}{\triangleq}
\newcommand{\nobarfrac}{\genfrac{}{}{0pt}{1}}
\DeclareMathOperator*{\essup}{ ess\,sup}
\newcommand{\1}{\mathds{1}}
\newcommand{\D}{\textup{d}}
\DeclareMathOperator{\avar}{AVaR}
\newcommand{\tref}[1]{\textup{\ref{#1}}}

\DefineFNsymbolsTM*{ARsymbols}{
  \textdagger    \dagger
  \textdaggerdbl \ddagger
  \textsection   \mathsection
  \textparagraph \mathparagraph
  \textbardbl    \|%
}%
\setfnsymbol{ARsymbols}

\newenvironment{tightitemize}{%
    \list{{\textup{$\bullet$}}}{\settowidth\labelwidth{{\textup{\qquad}}}
    \leftmargin\labelwidth \advance\leftmargin\labelsep
    \parsep 0pt plus 1pt minus 1pt \topsep 3pt \itemsep 3pt
    }}{\endlist}

\newtheorem{definition}{Definition}[section]

\newtheorem{theorem}{Theorem}[section]
\newtheorem{lemma}{Lemma}[section]

\newtheorem{assumption}{Assumption}[section]

\title{Subdifferentials of Convex Operators Valued in the Space of Integrable Functions with Application to Risk-Averse Optimization}
\author{Darinka Dentcheva\footnote{Stevens Institute of Technology, Department of Mathematical Sciences, Hoboken, NJ 07030, USA; email:
\texttt{darinka.dentcheva@stevens.edu}}\hspace{0.5em}  and
Andrzej Ruszczy\'nski\footnote{Department of Management Science and Information Systems,
Rutgers University,
Piscataway, NJ 08854, email:\texttt{rusz@rutgers.edu}}}

\begin{document}

\maketitle

\begin{abstract} We study differentiability properties of convex operators defined on a Banach space with values in an $\Lc_p$ space and of their compositions with monotonic convex functionals on this space.
We develop new tools for operators enjoying an additional feature known as the local property. The new approach and results go beyond the classical theory of normal integrands and lattice-valued operators. We further describe the subdifferentials of compositions of such operators with convex monotonic functionals.  The new results are applied to obtain novel optimality conditions in the subdifferential form for a broad class of risk-averse stochastic optimization problems with risk functionals as objectives, with partial information,
and with stochastic dominance constraints.
While our analysis is motivated by the theory and methods of risk-averse optimization, it addresses problems of a more general structure and has a potential for further applications.
\end{abstract}
%
%
%
%

\section*{Introduction}

Our study is motivated by stochastic optimization models with non-linear operators expressing risk aversion. The area of optimization integrating risk models has developed very rapidly in the recent years due to its mathematical challenges and practical relevance. Among the theoretical questions posed by risk models, a key question pertains to the subdifferentiability of operators valued in $\Lc_p$ spaces and their compositions with non-linear functionals. A lot of work containing fundamental results on subdifferentiation of integral functionals is available
in the area of convex analysis. We refer to the pioneering results in \cite{rockafellar1968integrals,Val69,rockafellar1971integrals,valadier1972sous,stra:65} and to the thorough exposition in the monographs \cite{cast:77,lev:85}. In this paper, we undertake the challenge to address subdifferentiability of convex operators $F:\Xf\to \Lc_p(\Omega,\Fc,P)$ defined on a Banach space $\Xf$ and composed further with nonlinear functionals. These structures arise in the modern directions of stochastic optimization, where further technical difficulties are posed by the non-linearity in probability.

Our goal is to establish subdifferentiability properties of such operators, describe their subdifferentials, and obtain optimality conditions in subdifferential form for stoch\-astic optimization problems incorporating risk models as objectives and constraints.

More precisely, we consider the following general optimization problem:
\begin{equation},
\label{general-risk-opt}
\begin{aligned}
\min \; & \varrho\big(F(x)\big),\\
        & \rho_t\big(G(x)\big)\leq 0, \quad \forall\ t\in J,\\
        & x\in\Yc.
\end{aligned}
\end{equation}
Here, $\Yc$ is a subset of a Banach space $\Xf$, and the operators $F:\Xf\to \Lc_p(\Omega,\Fc,P)$ and $G:\Xf\to \Lc_p(\Omega,\Fc,P)$ assign to a decision $x\in \Xf$ random variables $Z_F=F(x)$ and $Z_G=G(x)$, respectively. The functionals $\varrho:\Lc_p(\Omega,\Fc,P)\to \Rb$ and $\rho_t:\Lc_p(\Omega,\Fc,P)\to \Rb$ for all $t$ in an interval $J\subset\Rb$
are convex and nondecreasing with respect to the almost sure order.

We assume that the decision space $\Xf$ is a Banach space. Frequently, $\Xf=\Rb^n$, but it is convenient
to consider more general cases in view of applications to dynamic models.  We pay special attention to the case when
$\Xf = \Lc_{p'}(\Omega,\Gc,P;\Yf)$, where $p'\in [p,\infty]$ and $\Yf$ is a separable Banach space. The $\sigma$-subalgebra $\Gc \subseteq \Fc$ models the information available when the decision is made.
If $\Gc=\{\emptyset,\Omega)$ then $\Xf =\Yf$.

The mapping $F:\Xf\to\Lc_p(\Omega,\Fc,P)$, where $p\ge 1$,
describes the dependence of a random ``cost'' in a stochastic system on the decision variables $x\in \Xf$, that is,
$\big[F(x)\big](\omega)$ is the cost associated with decision $x$ and elementary event $\omega\in \Omega$.
Usually, the mapping $F(\Cdot)$ results from a composition of a cost function $c:\Yf \times \Df\to \Rb$ and a random data vector
$D: \Omega \to \Df$, where $\Df$ is a space of the data vector realizations. Then
$\big[F(x)\big](\omega) \eqdef c(x(\omega),D(\omega))$, $\omega\in \Omega$.
The mapping $G:\Xf\to\Lc_p(\Omega,\Fc,P)$ with $p\ge 1$ may represent a performance functional so that $\big[G(x)\big](\omega)$  measures a relevant system's feature for the elementary event $\omega\in \Omega$.

We call the  functionals $\varrho:\Lc_p(\Omega,\Fc,P)\to \Rb$ and $\rho_t:\Lc_p(\Omega,\Fc,P)\to \Rb$, which are convex and monotonic, \emph{risk functionals}. We point out that the mathematical expectation is a special case of a risk functional and our results generalize the results on subdifferentiability involving normal integrands.

While our results may be applicable to many areas, we note that the adopted structure accommodates optimization of coherent or convex measures of risk and risk constraints in the form of stochastic orders, which are of particular interest to us.
Modern theory of mathematical models of risk starts with the
first axiomatic proposals due to \cite{KijOhn:1993} in the context of mean--risk models. That set of axioms does not include the monotonicity axiom. Monotonicity with respect to stochastic dominance
was first proposed in \cite{OgRu:1999}. The complete sets of axioms are due to
\cite{ADEH:1999} for measures defined on finite probability spaces. It was extended to general spaces in \cite{Delbaen:2002,FolSch:2002,Leitner:2005,RuSh:2006a}.
Risk measures on the space of the quantile functions were considered in \cite{dentcheva2014risk}.
Further extensions include systemic measures of risk (see, e.g, \cite{Chen,kromer2016systemic,Burgert,ruschendorf2013mathematical,ekeland2011law,almen2023risk}) and measures on a probability space with a variable probability measure, called
risk forms \cite{dentcheva2020risk,dentcheva2023mini}.

Risk constraints based on stochastic orders were introduced in \cite{ddarsiam} for integer orders and further analyzed in \cite{ddarmp}. Optimality conditions in Lagrangian form for a basic problem formulation  were established first in \cite{ddarsiam}; for the inverse formulation of the second order stochastic dominance optimality conditions are provided in \cite{Dentcheva2006InverseSD}. Optimization with stochastic dominance constraints relates also to risk functionals; another form of optimality conditions showcasing this relation is established in \cite{dentcheva2008duality}. For relations to chance constraints or Average (Conditional) Value-at-Risk constraints, we refer to \cite{ddarport}.
In this paper, we shall provide a subdifferential form of optimality conditions for a general setting with inverse stochastic dominance constraints.

In successive sections, we address questions associated with the subdifferentiability of the compositions $\varrho\circ F$ and $\rho_t\circ G$ as well as and their implications for the optimality conditions for problem \eqref{general-risk-opt}.
In the first part of our paper, we focus on the  calculation of the subdifferential of the
composite risk function $\varphi(\Cdot)=\varrho(F(\Cdot))$, where $\varrho:\Lc_p(\Omega,\Fc,P)\to \Rb$ is a measure of risk.
In the second part of the paper, we analyze the compositions arising in the risk constraints in \eqref{general-risk-opt}.

In what follows, all equations and inequalities between elements of the space $\Lc_p(\Omega,\Fc,P)$, with $p\ge 1$, are understood in the almost sure sense.

\section{Composition of a Risk Functional with a Convex Operator}
\label{s:convex-operators}

First, we analyze the differentiability of the objective function and the constraint functions in problem \eqref{general-risk-opt}.
Convex lattice-valued operators and their subgradients were already considered by \cite{valadier1972sous}, and a thorough exposition can be found in the book \cite{lev:85}.
We recall the basic definitions and elementary properties.
\begin{definition}
\label{d:convex-mapping}
A mapping $F:\Xf\to\Lc_p(\Omega,\Fc,P)$ is \emph{convex} if for all $x,y\in \Xf$ and all $\alpha\in [0,1]$ we have
\begin{equation}
\label{convex-mapping}
 F(\alpha x + (1-\alpha) y) \le \alpha F(x) + (1-\alpha)F(y).
\end{equation}
\end{definition}
For a convex mapping $F:\Xf \to \Lc_p(\Omega,\Fc,P)$, at each $x\in \Xf$ and for any $h\in \Xf$, we can construct the differential quotients:
\begin{equation}
\label{Qtxh}
Q_t(x;h) \eqdef \frac{1}{t} \big[F({x}+th)-F({x})\big], \quad t >0.
\end{equation}
They are nondecreasing (in the lattice sense) functions of $t$, and satisfy for $t\in (0,1]$ the
inequalities
\[
F(x)- F(x-h) \le Q_t(x;h) \le F(x+h) - F(x).
\]
Thus the directional derivative
\begin{equation}
\label{F-prime}
F'(x;h) = \lim_{t\downarrow 0} Q_t(x;h), \quad h \in \Xf,
\end{equation}
is well-defined and belongs to $\Lc_p(\Omega,\Fc,P)$, as the limit in the strong and the order sense.
The mapping $F'(x;\Cdot)$ is convex and positively homogeneous by construction.
If $F(\Cdot)$ is continuous at $x$, then $F'(x;\Cdot)$ is continuous at 0.

Analogously to the scalar case, we define subgradients of vector-valued convex mappings.

\begin{definition}
\label{d:F-subgradient}
Suppose $F:\Xf\to \Lc_p(\Omega,\Fc,P)$ is a convex mapping. A continuous linear operator $S:\Xf \to \Lc_p(\Omega,\Fc,P)$ such that for all $h\in \Xf$
\begin{equation}
\label{subgradientL-def0}
Sh \le F(x+h) - F(x)
\end{equation}
is called the \emph{subgradient} of $F(\Cdot)$ at ${x}$. The set of all such operators is called the \emph{subdifferential}
of $F(\Cdot)$ at ${x}$ and is denoted by $\partial F(x)$.
\end{definition}
It is evident that $S\in \partial F(x)$ if and only if for all $h\in \Xf$
\begin{equation}
\label{subgradientL-def}
Sh \le F'(x;h).
\end{equation}
{ This characterization implies that
the set $\partial F(x)$ is convex and compact in the weak operator topology; see  \cite[Lem. 1.1]{lev:85}. The fact that
 $\Lc_p(\Omega,\Fc,P)$ is a topological complete vector lattice plays a key role.}

The following fundamental fact is due to \cite[Thm. 6]{valadier1972sous}.
\begin{theorem}
\label{t:sub_exists}
Suppose $F:\Xf\to \Lc_p(\Omega,\Fc,P)$ is convex and continuous at $x$. Then $\partial F(x)\ne \emptyset$ and for every $h\in \Xf$ we can find ${S}\in \partial F(x)$ such that $F'(x;h) = {S}h$.
\end{theorem}

We proceed now to new results in the setting of $\Lc_p$-spaces.

{ For an operator
$S:\Xf \to \Lc_p(\Omega,\Fc,P)$ and a function $\alpha \in \Lc_\infty(\Omega,\Fc,P)$ we define the operator
$\alpha S :\Xf \to \Lc_p(\Omega,\Fc,P)$ by $\big[(\alpha S)h\big](\omega) = \alpha(\omega)\big[(Sh)(\omega)\big]$, for $h\in \Xf$, $\omega \in \Omega$. If $S$ is continuous, so is $\alpha S$.}
 \begin{lemma}\label{l:generalized-convexity}
For all ${x}\in \Xf$, the set $\partial F(x)$ is convex in the following generalized sense: if $S_1,S_2\in \partial F(x)$
and $\alpha:\Omega\to[0,1]$ is measurable, then $\alpha S_1 + (1-\alpha) S_2\in \partial F(x)$ as well.
\end{lemma}
{ \begin{proof}
Using \eqref{subgradientL-def}, for all $h\in \Xf$ we obtain
\begin{multline*}
[\alpha S_1 + (1-\alpha) S_2 ]h  = [\alpha S_1]h + [(1-\alpha) S_2 ]h =
\alpha [S_1 h ] + (1-\alpha) [S_2 h ]  \\
\le \alpha F'(x;h) + (1-\alpha) F'(x;h) = F'(x;h),
\end{multline*}
and thus $\alpha S_1 + (1-\alpha) S_2 \in \partial F(x)$.
\end{proof}}

We pass now to the analysis of the composition
\begin{equation}
\label{phi-def}
\varphi(x)\eqdef \varrho\big(F(x)\big), \quad x\in \Xf,
\end{equation}
where $F:\Xf\to \Lc_p(\Omega,\Fc,P)$ is a convex operator, and $\varrho: \Lc_p(\Omega,\Fc,P)\to \Rb$ is a risk functional. We say that a risk functional is
nondecreasing, if $Z \le V$ implies that $\varrho(Z) \le \varrho(V)$. By \cite[Prop. 3.1]{RuSh:2006a}, a convex and nondecreasing risk functional
on $\Lc_p(\Omega,\Fc,P)$, where $p\in [1,\infty]$, is subdifferentiable everywhere.
It can be easily verified that if  the mapping  $F:\Xf\to\Lc_p(\Omega,\Fc,P)$ is convex and  $\varrho:\Lc_p(\Omega,\Fc,P)\to {\Rb}$
is convex and nondecreasing, then  the composite  function
$\varrho\circ F$ is convex.
To calculate its subdifferential, we first consider
its directional derivatives.

\begin{lemma}
\label{l:phi-dir-der}
Suppose the operator $F:\Xf\to \Lc_p(\Omega,\Fc,P)$ is convex and continuous at $x$, and the risk functional $\varrho:\Lc_p(\Omega,\Fc,P) \to {\Rb}$
is convex and nondecreasing. Then for all $x,h\in\Xf$
\begin{equation}
\label{phi-prime}
\varphi'({x};h) = \varrho'\big(F({x});F'({x};h)\big).
\end{equation}
\end{lemma}
\begin{proof}
The proof uses a standard argument, but we provide it for completeness. Recalling the definition \eqref{F-prime}, for all $\tau\ge 0$ we have
\[
F({x}+\tau h)  = F({x}) + \tau F'({x};h) + \tau r(\tau,h),
\]
with $r(\tau,h)\downarrow 0$ when $\tau\downarrow 0$. Furthermore, by virtue of Theorem \cite[Prop. 3.1]{RuSh:2006a}, $\varrho(\Cdot)$ is
continuous and subdifferentiable everywhere, and for any $d\in \Lc_p(\Omega,\Fc,P)$ we have
\[
\varrho\big(F({x})+ \tau d\big) = \varrho\big(F({x})\big) + \tau \varrho'\big(F({x});d\big) + \tau R(\tau,d),
\]
where $R(\tau,d)\downarrow 0$ when $\tau\downarrow 0$. Suppose $0<\tau \le \bar{\tau}$. Since $r(\Cdot,h)$ is nondecreasing,
the two expansions yield
\begin{align*}
\lefteqn{\varphi({x}+\tau h) - \varphi({x}) = \varrho\big(F({x}+\tau h)\big) - \varrho\big(F({x})\big)}\quad \\
&=\varrho\big(F({x}) + \tau F'({x};h) + \tau r(\tau,h)\big) - \varrho\big(F({x})\big)\\
&\le\varrho\big(F({x}) + \tau F'({x};h) + \tau r(\bar{\tau},h)\big) - \varrho\big(F({x})\big)\\
&= \tau \varrho'\big(F({x}); F'({x};h) +  r(\bar{\tau},h)\big) + \tau R\big(\tau,F'({x};h) +  r(\bar{\tau},h)\big).
\end{align*}
The directional derivative $\varrho'\big(F(x); \Cdot)$ is convex and positively homogeneous, hence subadditive, and thus
\begin{multline*}
\varphi({x}+\tau h) - \varphi({x}) \\
\le
\tau \varrho'\big(F({x}); F'({x};h)\big) + \tau  \Big[ \varrho'\big(F({x});r(\bar{\tau},h)\big) +R\big(\tau,F'({x};h) +  r(\bar{\tau},h)\big) \Big].
\end{multline*}
On the other hand, the convexity of both functions and the monotonicity of $\varrho(\Cdot)$ imply that
\[
\varphi({x}+\tau h) - \varphi({x}) \ge \varrho\big(F({x}) + \tau F'({x};h)\big) - \varrho\big(F({x})\big)
\ge \tau \varrho'\big(F({x}); F'({x};h)\big).
\]
Combining the last two inequalities and dividing by $\tau$ we obtain
\begin{multline*}
0 \le \frac{1}{\tau} \big[ \varphi({x}+\tau h) - \varphi({x})\big] - \varrho'\big(F({x}); F'({x};h)\big)\\
\le  \varrho'\big(F({x});r(\bar{\tau},h)\big) +R\big(\tau,F'({x};h) +  r(\bar{\tau},h)\big).
\end{multline*}
After passing to the limit with $\tau\downarrow 0$, we get
\[
0 \le \varphi'({x};h) - \varrho'\big(F({x}); F'({x};h)\big) \le \varrho'\big(F({x});r(\bar{\tau},h)\big).
\]
As $\varrho'\big(F({x});r(\bar{\tau},h)\big)\downarrow 0$, when $\bar{\tau}\downarrow 0$, equation \eqref{phi-prime} is true.
\end{proof}

This allows us to calculate the subdifferential of the composition. { The following result is known; we may refer here
to \cite{lemaire1985application}, \cite[Thm. 1.9]{lev:85}, and \cite{combari1994sous}. We provide a version suitable for
our risk measure applications with a simple proof, which
improves on \cite[Thm. 3.3]{RuSh:2006a} and \cite[Thm 6.14]{shapiro2021lectures}.}
\begin{theorem}
\label{t:composition_subdif}
Suppose the operator $F:\Xf\to \Lc_p(\Omega,\Fc,P)$, where $\Xf$ is a Banach space and $p\in [1,\infty)$, is convex and continuous at $x$, and the risk functional $\varrho:\Lc_p(\Omega,\Fc,P) \to {\Rb}$
is convex and nondecreasing.
Then the composite function $\varphi=\varrho\circ F$ is subdifferentiable at ${x}$ and
\begin{equation}
\label{composition_subdif}
\partial \varphi ({x}) =
 \bigcup_{\nobarfrac{\zeta\in\partial\varrho (F({x}))}{S \in \partial F(x)}} S^*\zeta.
\end{equation}
\end{theorem}
\begin{proof}
Using Lemma \ref{l:phi-dir-der} and Theorem \ref{t:sub_exists}, we obtain the equation
\[
\varphi'({x};h)
=\max_{\zeta\in \partial \varrho(F({x}))} \Eb\big[ \zeta\,F'({x};h)\big]
= \max_{\zeta\in \partial \varrho(F({x}))} \Eb\big[ \zeta\,\essup_{S\in \partial F(x)} Sh\big].
\]
Since every $\zeta\in \partial \varrho(F({x}))$ is nonnegative, and the ``$\essup$'' above is attained, we can continue this chain of equations as follows:
\begin{multline}
\label{phi-prime-supp-g}
\varphi'({x};h)= \max_{\zeta\in \partial \varrho(F({x}))} \max_{S\in \partial F(x)} \Eb\big[ \zeta\,Sh\big]\\
= \max_{\zeta\in \partial \varrho(F({x}))} \max_{S\in \partial F(x)} \langle S^*\zeta,h \rangle
= \max_{g \in \Dc}\, \langle g,h \rangle,\qquad
\end{multline}
where $\Dc$ is the set on the right hand side of \eqref{composition_subdif}.

We shal verify that $\Dc$ is convex. Suppose $g_1 = S_1^*\zeta_1^{\phantom{*}}$, $g_2 = S_2^*\zeta_2^{\phantom{*}}$, where $\zeta_1,\zeta_2\in  \partial \varrho(F({x}))$, and $S_1,S_2\in \partial F({x})$.
 For $0<\alpha <1$ we verify whether $g = \alpha g_1 + (1-\alpha)g_2 \in \Dc$.
By the convexity of $\partial \varrho(\Cdot)$,
\[
\zeta = \alpha \zeta_1 + (1-\alpha)\zeta_2\in \partial\varrho(F({x})).
\]
Furthermore,
\[
g = \big( \theta  S_1+  (1-\theta) S_2\big)^* \zeta,
\]
with (for $\zeta(\omega)>0$)
\[
\theta(\omega) = \frac{\alpha \zeta_1(\omega)}{\zeta(\omega)},\quad \omega \in \Omega.
\]
If $\zeta(\omega) = 0$, then both $\zeta_1(\omega)=0$ and $\zeta_2(\omega)=0$, and we can set $\theta(\omega)=0$. In any case,
since $\zeta_1,\zeta_2 \ge 0$, we have $\theta\in [0,1]$ a.s.. The generalized convexity of $\partial F({x})$ (see Lemma \ref{l:generalized-convexity}) entails
\[
S = \theta  S_1+  (1-\theta) S_2 \in \partial F({x}),
\]
and thus $g\in \Dc$. Therefore, the set $\Dc$ is convex.

The closedness of $\Dc$ follows from the fact that for every $h\in \Xf$ the maximum of $\langle g, h\rangle$ over $g \in \Dc$ is attained in \eqref{phi-prime-supp-g}. Denote a maximizer by $g(h)$.
If a limit $\bar{g}$ of a sequence $\{g^k\}_{k\in \Nb}\subset \Dc$ was not an element of $\Dc$, then it could be  separated from $\Dc$ by a linear functional $\bar{h}$:
\[
\langle \bar{g},\bar{h}\rangle > \langle g,\bar{h} \rangle, \quad \forall\, g\in \Dc.
\]
This would lead to a contradiction:
\[
\langle \bar{g},\bar{h}\rangle > \langle g(\bar{h}),\bar{h} \rangle = \max_{g\in \Dc} \;\langle g,\bar{h} \rangle \ge \lim_{k\to\infty} \langle g^k,\bar{h} \rangle  = \langle \bar{g},\bar{h}\rangle.
\]
Summing up, it follows from \eqref{phi-prime-supp-g} that $\varphi'({x};\Cdot)$ is the support function of a convex closed set $\Dc$, and thus
$\Dc = \partial \varphi'({x};0) = \partial \varphi(x)$.
\end{proof}

Further analysis of the composition $\rho\,\circ F$ is dependent on the structure of the mapping $F(\Cdot)$. We address this
issue in the next subsection.

\section{Composition of a Risk Functional with a Convex Integrand}
\label{s:comp-convex-integrand}

In this section, $\Xf = \Lc_{p'}(\Omega,\Fc,P;\Yf)$, where $\Yf$ is a  separable Banach space,
  and $p\in [1,\infty)$,  $p'\in [p,\infty]$.
We include the case of $p'=\infty$ in view of the application to the analysis in the next section. In some intermediate results, we also allow $p=\infty$.

We assume that the mapping $F:\Lc_{p'}(\Omega,\Fc,P;\Yf)\to \Lc_p(\Omega,\Fc,P)$ has the following structure:
\begin{equation}
\label{F-integrand}
\big[F(x)\big] (\omega) = f(x(\omega),\omega),\quad \omega\in \Omega,
\end{equation}
where $f:\Yf\times\Omega\to \Rb$ is a convex and continuous function of the first argument, for almost all values of the second argument, and measurable with respect to the second argument,
for all values of the first argument. Such functions are called \emph{convex integrands}.

The mapping $F(\Cdot)$ is convex in the sense of Definition \ref{d:convex-mapping}, but
our setting is more specific because it excludes a null set \emph{before} all $x,y\in \Xf$ are
considered. This allows for an explicit description of the subgradients of $F(\Cdot)$.

The differential quotients \eqref{Qtxh} of the mapping \eqref{F-integrand} have a specific form: for almost all $\omega\in \Omega$
\[
\big[ Q_t(x;h)\big](\omega) = \frac{1}{t} \big[f\big(x(\omega)+th(\omega), \omega\big)-f\big(x(\omega),\omega\big)\big], \quad t >0.
\]
By the convexity of $f(\Cdot,\omega)$ we conclude that for all $h\in \Xf$ and for almost all $\omega\in \Omega$
\begin{equation}
\label{F-prime-integrand}
\big[F'(x;h)\big](\omega)= \lim_{t\downarrow 0} \big[ Q_t(x;h)\big](\omega)= f'\big(x(\omega),\omega; h(\omega)\big).
\end{equation}

We can now establish the decomposable structure of the subdifferential.

\begin{lemma}
\label{l:local}
If $L \in \partial F(x)$ then $L$ is \emph{local} in the following sense: for all $B\in \Fc$ and all $h\in \Xf$ we have
\[
L(\1_B h) = \1_B Lh.
\]
\end{lemma}
\begin{proof}
For all events $B\in \Fc$ and all $h\in {\Xf}$ we have $F'(x;\1_B h) = \1_B F'(x;h)$, and thus
every $L\in \partial F(x)$ satisfies the inequalities:
\begin{equation}
\label{L1B}
-\1_B F'(x;-h) \le L(\1_B h) \le \1_B F'(x;h).
\end{equation}
Consequently, $ L(\1_B h)$ vanishes outside of $B$.
As $L$ is linear,
\[
\1_B L(\1_{B}h) =\1_B  Lh - \1_B L(\1_{B^c} h) = \1_B  Lh,
\]
because the second term disappears due to \eqref{L1B} for $B^c$. Therefore, $ L(\1_B h)$ coincides with $Lh$ on $B$.
\end{proof}

The local property is essential in the following theorem. We draw the attention of the Reader that our technique
using the Radon-Nikodym theorem for vector measures, covers the notoriously difficult case of $p'=\infty$.
{ The latter case is less difficult than usual due to the continuity assumption
about the integrand $f(\Cdot,\omega)$, which does not take
the value $+\infty$.}

\begin{theorem}
\label{t:F-subgradient-structure-integrand}
Suppose $F:\Lc_{p'}(\Omega,\Fc,P;\Yf)\to \Lc_p(\Omega,\Fc,P)$ is given by \eqref{F-integrand}, $1 \le p\le p' \le\infty$,
$\Yf$ is a  separable Banach space, 
$f(\Cdot,\omega)$ is convex and continuous at $x(\omega)$ for almost all $\omega\in\Omega$.
Then $S\in \partial F(x)$ if and only if a function $s\in \Lc_{r}(\Omega,\Fc,P;\Yf^*)$ exists, with $1/p'+1/r=1/p$, such that
for almost all $\omega\in \Omega$ we have $s(\omega)\in \partial f(x(\omega),\omega)$,
and for all $h\in \Lc_{p'}(\Omega,\Fc,P;\Yf)$
\begin{equation}
\label{F-sub-dec}
\big[Sh\big](\Cdot) = \langle s(\Cdot),h(\Cdot) \rangle,\quad \text{a.s.}.
\end{equation}
\end{theorem}
\begin{proof}
Every operator of the form \eqref{F-sub-dec}, with a measurable selector $s(\Cdot)\in \partial f(x(\Cdot),\Cdot)$,
is a subgradient of $F$ at $x$, because for  all
$h\in \Xf$
\[
\langle s(\omega),h(\omega)\rangle \le f'\big(x(\omega),\omega;h(\omega)\big) = \big[F'(x;h)\big](\omega),
\]
for almost all $\omega\in \Omega$,
and thus Definition \ref{d:F-subgradient} is satisfied. Since $s\in \Lc_{r}(\Omega,\Fc,P;\Yf^*)$, in view of the H\"older inequality,
the operator \eqref{F-sub-dec} is a continuous
mapping from $\Lc_{p'}(\Omega,\Fc,P;\Yf)$ to $\Lc_p(\Omega,\Fc,P)$.

The difficult part is to prove that no other subgradients exist. Suppose $L\in \partial F(x)$. Consider the functional $\ell:\Xf \to \Rb$
defined as
\[
\ell(h) = \Eb[Lh]    ,\quad  h\in \Xf.
\]
As $L$ is linear and continuous,  so is $\ell(\Cdot)$.
Due to the local property of $L$ established in Lemma \ref{l:local},
\begin{equation}
\label{ell-local}
\ell(\1_B h) = \Eb\big[L(\1_Bh)\big] = \Eb\big[\1_B Lh\big] , \quad \forall B\in \Fc.
\end{equation}
Take any $y\in \Yf$ and set a constant $h=\1_{\Omega} y$. Then
\[
\ell(\1_B y) = \Eb \big[ \1_BL(\1_\Omega y)\big].
\]
Since the functional $y \mapsto \ell(\1_B y)$ is continuous, then  $M(B) \in \Yf^*$ exists, such that
\[
\Eb \big[ \1_BL(\1_\Omega y)\big] = \langle M(B), y \rangle, \quad \forall\,y \in \Yf.
\]
It follows from the last equation that the function $B \mapsto M(B)$ is a $\Yf^*$-valued, finite, and countably additive vector measure on $\Fc$, which is absolutely continuous
with respect to $P$. By virtue of the
Radon-Nikodym theorem for vector measures \cite[Ch. III]{diestel1977}, a measurable function $v:\Omega \to \Yf^*$ exists, such that
\[
M(B) = \int_B v(\omega)\; P(\D\omega), \quad \forall\,B \in \Fc.
\]
Combining the last three equations, we obtain
\[
\ell(\1_B y) = \int_\Omega \big\langle v(\omega), \1_B(\omega) y \big\rangle \; P(\D\omega), \quad \forall\,B \in \Fc.
\]
For any simple function,
\[
h = \sum_{n=1}^N \1_{B_n} y_n,
\]
using the linearity of $\ell(\Cdot)$, we get
\[
\ell(h) = \sum_{n=1}^M \int_\Omega \big\langle v(\omega), \1_{B_n}(\omega) y_n \big\rangle \; P(\D\omega) = \int_\Omega \big\langle v(\omega), h(\omega) \big\rangle \; P(\D\omega).
\]
Therefore,
\[
\ell(h) = \Eb\big[ \langle v(\Cdot),h(\Cdot) \rangle\big], \quad  \forall\,h\in \Xf.
\]
This, combined with \eqref{ell-local}, entails the equation
\[
\Eb\big[\1_B Lh\big] =  \Eb\big[ \1_B \langle v(\Cdot),  h(\Cdot) \rangle\big], \quad \forall B\in \Fc.
\]
Thus, for all $h \in \Xf$,
\begin{equation}
[Lh](\Cdot) = \langle v(\Cdot),  h(\Cdot) \rangle, \quad \text{a.s.}. \notag
\end{equation}
This representation implies that the function $\langle v(\Cdot),  h(\Cdot) \rangle$ is an element
of $\Lc_p(\Omega,\Fc,P)$, for all $h \in \Lc_{p'}(\Omega,\Fc,P;\Yf)$. Therefore, $v \in \Lc_{r}(\Omega,\Fc,P;\Yf^*)$.

Let
\[
\Sc = \big\{ s\in \Lc_{r}(\Omega,\Fc,P;\Yf^*): s(\Cdot) \in \partial f(x(\Cdot),\Cdot)\big\}.
\]
Suppose $v\notin \Sc$. Because $\Sc$ is convex and closed, by virtue of the separation theorem, a function $\bar{h}\in \Xf$
and a real number $\varepsilon>0$ exist, such that
for all $s\in \Sc$
\begin{equation}
\label{ell-separation}
\Eb\big[ \langle v, \bar{h} \rangle \big] \ge \Eb\big[ \langle s,\bar{h} \rangle \big] + \varepsilon.
\end{equation}
Let $\bar{s}$ be an element of $\Sc$ such that
\begin{equation}
\label{bars-separation}
\langle \bar{s}(\Cdot),\bar{h}(\Cdot) \rangle = f'\big(x(\Cdot),\Cdot;\bar{h}(\Cdot)\big);
\end{equation}
such a function can be found, owing to the measurable selection theorem.
Indeed, for almost all $\omega\in \Omega$ we have
\[
f'({x}(\omega),\omega;\bar{h}(\omega)) = \max_{s(\omega)\in \partial f(\bar{x}(\omega),\omega)} \langle s(\omega),\bar{h}(\omega)\rangle;
\]
the maximum is attained at some $\bar{s}(\omega)$, due to the weak$^*$ compactness of $\partial f(\bar{x}(\omega),\omega)$.
The  relations \eqref{ell-separation}--\eqref{bars-separation} imply that
\[
\Eb\big[ L \bar{h}\big]  = \Eb\big[ \langle v,  h \rangle \big] \ge \Eb\big[ F'(x;\bar{h})\big] + \varepsilon.
\]
which contradicts \eqref{subgradientL-def}. Consequently, $v \in \Sc$.
\end{proof}

We can now calculate the subdifferential of the composition $\varphi(\Cdot) = \varrho(F(\Cdot))$, where $\varrho(\Cdot)$ is a risk functional.

\begin{theorem}
\label{t:risk-Strassen-integrand}
Suppose $F:\Lc_{p'}(\Omega,\Fc,P;\Yf)\to\Lc_p(\Omega,\Fc,P)$ is given by \eqref{F-integrand}, with $p\in [1,\infty)$ and $p'\in [p,\infty]$,
$f(\Cdot,\omega)$ is convex and continuous at $x(\omega)$ for almost all $\omega\in\Omega$, $F(\Cdot)$ is continuous at $x$, and $\Yf$ is a  separable Banach space.
Further, suppose
$\varrho:\Lc_p(\Omega,\Fc,P)\to\Rb$  is convex and nondecreasing.
Then the composite function $\varphi=\varrho\circ F$ is subdifferentiable at ${x}$ and
\begin{equation}
\label{dif-rho-integrant}
\partial \varphi ({x}) =
  \bigcup_{\zeta\in\partial\varrho (F({x}))} \hspace{-1em}
  \big\{ { g\in\Lc_{q'}(\Omega,\Fc,P;\Yf^*)}: g(\omega) \in \zeta (\omega)\,\partial f({x}(\omega),\omega),\  \text{a.s.}\big\},
\end{equation}
where $1/p'+ 1/q' = 1$.
\end{theorem}
\begin{proof}
We use the representation of the subdifferential established in Theorem~\ref{t:composition_subdif} and the characterizations of the
subgradients of $F(\Cdot)$ established in Theorem~\ref{t:F-subgradient-structure-integrand}. For any $h\in \Xf$, any $\zeta\in \partial \varrho(F(x))$
and any $S\in \partial F(x)$ we have
\[
\langle S^*\zeta,h \rangle = \langle \zeta,Sh \rangle = \int_\Omega  \big\langle \zeta(\omega) s(\omega), h(\omega) \big \rangle \;P(\D\omega),
\]
where $s(\Cdot) \in \partial f(x(\Cdot),\Cdot)$. This verifies \eqref{dif-rho-integrant}.
\end{proof}

\section{The Case of Partial Information}
\label{sub-convex-composition}

We now assume that $\Xf = \Lc_{p'}(\Omega,\Gc,P;\Yf)$, where $\Yf$ is  a separable Banach space, and $\Gc \subseteq \Fc$ represents the information available
when the decision $x$ is made. As before, $p\in [1,\infty)$,  $p'\in [p,\infty]$. The special case of $p'=\infty$ and $\Gc = \{ \emptyset,\Omega\}$ covers the case
when $x$ is a deterministic element of $\Yf$.

The mapping
$F:\Xf\to \Lc_p(\Omega,\Fc,P)$ is defined by a convex integrand \eqref{F-integrand}.
Again, the mapping $F(\Cdot)$ is convex and we plan to analyze the structure of its subdifferential. The directional derivatives
of $F(\Cdot)$ still have the form \eqref{F-prime-integrand}.

The following result extends Theorem \ref{t:F-subgradient-structure-integrand} to the case of partial information.

\begin{theorem}
\label{t:F-subgradient-structure}
Suppose $F:\Xf\to \Lc_p(\Omega,\Fc,P)$, with $p\in [1,\infty]$, is given by \eqref{F-integrand},
$\Yf$ is a  separable Banach space, $f(\Cdot,\omega)$ is convex and continuous at $x(\omega)$ for almost all $\omega\in\Omega$,
and $F(\Cdot)$ is continuous at $x$ as well.
Then $ S\in \partial F(x)$ if and only if a function $s\in \Lc_r(\Omega,\Fc,P;\Yf^*)$ exists, with $1/p'+1/r=1/p$, such that
$s(\omega)\in \partial f(x(\omega),\omega)$ for almost all $\omega\in \Omega$,
and for all $h\in \Xf$
\begin{equation}
\label{F-sub-dec-1}
\big[Sh\big](\Cdot) = \big\langle s(\Cdot),h(\Cdot) \big\rangle,\quad \text{a.s.}.
\end{equation}
\end{theorem}
\begin{proof}
As in Theorem \ref{t:F-subgradient-structure-integrand}, we can verify that every operator $S$ of the form \eqref{F-sub-dec-1} is a subgradient of $F(\Cdot)$. We need to prove that no other subgradients exist.

Our idea is to allow $x$ and $h$ to by $\Fc$-measurable and to leverage Theorem \ref{t:F-subgradient-structure-integrand}.
Define the space $\Xb = \Lc_{p'}(\Omega,\Fc,P;\Yf)$ and denote by  $\mathbbm{x}$ and  $\mathbbm{h}$ its elements.
Let the mapping $\Fb:\Xb\to \Lc_p(\Omega,\Fc,P)$ be still defined as in \eqref{F-integrand}.

Suppose $\xb\in \Xb$ is equal to $x$ a.s..
The function $\Fb'(\xb;\Cdot)$ is convex, positively homogeneous, and continuous on $\Xb$.
Suppose $S\in \partial F(x)$. Then $Sh\le F'(x;h)$ for all  $h\in \Xf$. This means that $S(\Cdot)$ is a minorant of $\Fb'(\xb;\Cdot)$
on the subspace of $\Gc$-measurable functions $\Nc=\big\{\hb\in \Xb: \Eb[\hb|\Gc] = \hb \big\}$. By the
generalization of the Hahn--Banach theorem to lattice-valued operators (see, \cite[Thm. X.5.7]{kantorovich1984functional}), $S$ can be extended to a linear functional $\Sb$ on $\Xb$
which is a minorant of $\Fb'(\xb;\Cdot)$
and which coincides with $S$ on $\Nc$. As  $\Fb'(\xb;\Cdot)$ is continuous,  $\Sb\in \partial \Fb(\xb)$.
By Theorem \ref{t:F-subgradient-structure-integrand} at the point $\xb$,
a function $s\in \Lc_r(\Omega,\Fc,P;\Yf^*)$ exists, such that
for almost all $\omega\in \Omega$ we have $s(\omega)\in \partial f(x(\omega),\omega)$,
and for all $\hb\in \Xb$
\[
\Sb\hb = \langle s,\hb \rangle,\quad \text{a.s.}.
\]
On the subspace $\Nc$, the above formula reduces to \eqref{F-sub-dec-1}.
\end{proof}

The subdifferential of $\varphi(x) = \varrho(F(x))$ can be calculated in a similar way as in Theorem \ref{t:risk-Strassen-integrand}.

\begin{theorem}
\label{t:pr-subdifa}
Suppose $\Yf$ is a  separable Banach space, $F:\Xf\to\Lc_p(\Omega,\Fc,P)$ is given by \eqref{F-integrand}, $p\in [1,\infty)$,
$f(\Cdot,\omega)$ is convex and continuous at $x$ for almost all $\omega\in\Omega$,
$F(\Cdot)$ is continuous at $x$ as well, and
$\varrho:\Lc_p(\Omega,\Fc,P)\to\Rb$  is convex and nondecreasing.
Then the composite function $\varphi=\varrho\circ F$ is subdifferentiable at ${x}$ and\footnote{The symbol ``$\lessdot$'' means ``is a
measurable selector of.''}
\begin{equation}
\label{dif-strass}
\partial \varphi ({x})  = \bigcup_{\nobarfrac{\zeta\in \partial\varrho (F({x}))}{s(\cdot)\lessdot \partial f(x(\cdot),\cdot)}} \Eb[\zeta s|\Gc].
\end{equation}
\end{theorem}
\begin{proof}
We use the representation of the subdifferential established in Theorem~\ref{t:composition_subdif} and the characterizations of the
subgradients of $F(\Cdot)$ established in Theorem~\ref{t:pr-subdifa}.
For all $h\in \Xf$, $\zeta\in \partial \varrho(F(x))$,
and $S\in \partial F(x)$, we have
\[
\langle S^*\zeta,h \rangle = \langle \zeta,Sh \rangle = \Eb\big[ \zeta \langle s, h \rangle \big] = \Eb\big[\big\langle  \Eb[\zeta s|\Gc] , h \rangle \big],
\]
where $s(\omega) \in \partial f(x(\omega),\omega)$. This verifies \eqref{dif-strass}.
\end{proof}

Our results, in a special case, allow to obtain the famous subdifferential disintegration result, which is known as the \emph{Strassen Theorem}.
In the case of a positively homogeneous $f(\Cdot,\omega)$ and $\varrho(\Cdot) \equiv \Eb[\Cdot]$,
it was discovered in \cite{stra:65}.
Generalizations to convex functions and convex integrands are provided in
\cite{rockafellar1968integrals,rockafellar1971integrals,ioffe1972subdifferentials} and the books \cite{cast:77,lev:85}.
\begin{theorem}
\label{t:Strassen}
Suppose $\Xf$ is a  separable Banach space, $F:\Xf\to\Lc_1(\Omega,\Fc,P)$ is given by $F(x,\omega) = f(x,\omega)$, $\omega\in \Omega$,
$f(\Cdot,\omega)$ is convex and continuous at $x$ for almost all $\omega\in\Omega$,
and $ F(\Cdot)$ is continuous at $x$ as well.
Then the function $\Eb[ F(\Cdot)]$ is subdifferentiable at ${x}$ and
\[
\partial \Eb[ F(x)]=\Eb[\partial F(x)]=
\int_\Omega \partial f({x},\omega)\; P(\D\omega).
\]
\end{theorem}

The result follows from setting $p=1$, $p'=\infty$, $\Gc = \{ \emptyset, \Omega\}$, and $\varrho(\Cdot) \equiv \Eb[\Cdot]$ in Theorem \ref{t:pr-subdifa}.

\section{Stochastic Dominance Operators}

We now turn to convex operators associated with stochastic dominance relations.

To the best of our knowledge the first general notion of a stochastic order was introduced in \cite{Lorenz}. It is related to the theory of weak majorization, as discussed in \cite{arnold2012majorization}. A comparison of random variables with respect to the first-order stochastic dominance is used in the context of statistical tests in the early works \cite{MannW,Blackwell1951ComparisonOE,lehmann1952testing}.
The stochastic dominance relations have received considerable attention after the works \cite{QS,1965PCFD,hanoch1969efficiency,Bawa1975OPTIMALRF,hadar1969rules,rothschild1978increasing}. The characterization of the first and second-order dominance via families of (differentiable) utility functions was stated in \cite{hadar1969rules} and rigorously proved in \cite{Bawa1975OPTIMALRF}, where also the third-order dominance is characterized. Stochastic dominance of fractional order $p\geq 1$ was introduced for non-negative bounded random variables in \cite{Fishburn1976ContinuaOS}, where also the monotonicity of the relation was proved.
A thorough survey and analysis of stochastic order relations is contained in the monographs \cite{MuellerStoyan} and \cite{ShaSha}. Below, we provide the information about fractional orders that is relevant to our research.

Consider a random variable $Z \in \Lc_1(\Omega,\Fc,P)$ with $p\in [1,\infty]$ and its distribution function $H_Z(\eta) =  P[ Z \le \eta ]$ for $\eta \in \Rb$. We define the integrated distribution function as follows:
\begin{equation}
\label{second}
 H^{(2)}_Z(\eta) = \int_{-\infty}^{\eta} H_Z(\alpha) \ \D\alpha
=
 \Eb  \big[ \max ( \eta - Z, 0 )\big] \quad \mbox{for} \  \eta \in \Rb.
\end{equation}

Recall that the left-continuous inverse of the cumulative distribution function
$H_Z (\Cdot)$ is defined as follows:
\[
 H^{(-1)}_Z(p) = \inf\ \{ \eta : H_Z(\eta) \ge p \}
\quad\text{for all }\; 0 < p < 1.
\]
Given $p \in [0,1]$, the number $q=q_Z(p)$ is called a \emph{$p$-quantile}
of the random variable $X$ if
\[
 P ( Z < q ) \le p \le  P (Z\le q ) .
\]
For $p\in (0,1)$, the set of $p$-quantiles is a closed interval and
$ H^{(-1)}_Z (p) $ represents its left end. We adopt the convention
that  $H^{(-1)}_Z(1) = +\infty$, if the $1$-quantile does not exist.

The integrated quantile function $H^{(-2)}_Z : \Rb \rightarrow {\Rb }$
is defined as
\begin{equation}
\label{F-2}
  H^{(-2)}_Z(p) =\int_0^p H^{(-1)}_Z(\alpha)\; \D\alpha  \quad \mbox{for } \;  0 < p \le 1,
\end{equation}
We also define $H^{(-2)}_Z(0) = 0$ and $H^{(-2)}_Z(p) = +\infty$ for all  real numbers $p\not\in[0,1]$.
The function $H^{(-2)}_Z(\Cdot)$  is
the absolute Lorenz function introduced in \cite{lorenz1905methods}.
It is well defined for any
random variable $Z\in \Lc_1(\Omega,\Fc,P).$ It is convex as an integral of a non-decreasing function.

The Fenchel duality relation between the integrated quantile function
$H^{(-2)}_Z(\Cdot)$ and the integrated distribution function
$H^{(2)}_Z(\Cdot)$ has been established in \cite{ruzsdom}.
\begin{theorem}
\label{t:Fduality}
For every integrable random variable $Z$, we have
\[
 H^{(-2)}_Z =  \big[H^{(2)}_Z\big]^*\quad\text{and}\quad
H^{(2)}_Z =  \big[H^{(-2)}_Z\big]^*.
\]
\end{theorem}
\begin{definition}
\label{d:sdp}
A random variable $X$ dominates in the first order the random variable $Y$ if
\begin{equation}
\label{d:fsd}
H_X(\eta) \le H_Y(\eta)\;\; \forall \eta\in \Rb\Leftrightarrow  H^{(-1)}_X (p)\geq H^{(-1)}_Y (p)
\;\; \forall p\in (0,1).
\end{equation}
For $X,Y\in \Lc_1(\Omega,\Fc,P)$, $X$ dominates in the second order the
random variable $Y$ if
\begin{equation}
\label{d:ssd}
H^{(2)}_X(\eta) \le H^{(2)}_Y(\eta)\;\;\forall \eta\in \Rb \Leftrightarrow  H^{(-2)}_X (p)\geq H^{(-2)}_Y (p)
\;\; \forall p\in (0,1].
\end{equation}
\end{definition}
The equivalence of the two requirements in \eqref{d:fsd} is obvious while the equivalence of the two conditions in \eqref{d:ssd} follows from Theorem~\ref{t:Fduality}.
In our analysis, we shall focus on the second order dominance relation expressed as a relation between the Lorenz functions of the random variable in question.

Let $J=[\alpha,\beta]\subset (0,1].$ We shall relax the second-order stochastic dominance relation to the interval $J$, i.e.,  for two integrable random variables $X$ and $Y$, it is said that \emph{$X$ is larger than the random variable $Y$ w.r.to the Lorenz dominance on $J$}  if
\begin{equation}
\label{d:ssd_onJ}
H^{(-2)}_X(\eta) \geq H^{(-2)}_Y(\eta)\quad \text{for all }  p\in J.
\end{equation}
We define the set
\begin{equation}
\label{d:setB}
\mathbb{B}(Y,J) = \big\{X\in \Lc_1(\Omega,\Fc,P):
H^{(-2)}_X(p) \ge H^{(-2)}_Y(p)\ \, \forall\  p \in J\big\}.
\end{equation}
The following theorem establishes relevant properties of the set $B(Y,J)$. Statements (a) and (c) were shown in \cite{DDARbook}; we provide the whole proof for the sake of clarity and convenience of the reader.
\begin{lemma}
\label{l:F-2_posithom_convex}\phantom{X}
\begin{tightitemize}
\item[\textup{(a)}] For every $p\in (0,1)$, the mapping $Z\mapsto H^{(-2)}_Z(p)$  is continuous, concave
and positively homogeneous on $\Lc_1(\Omega,{\mathcal F},P)$.
\item[\textup{(b)}] The concave subdifferential of the mapping $Z\mapsto H^{(-2)}_Z(p)$ has the form
\begin{equation}
\label{e:Lorenz-subdifferential}
\begin{aligned}
  \partial_Z H^{(-2)}_Z(p) =  \Big\{p\zeta\in\Lc_\infty (\Omega,\Fc,P): ~ &\Eb[\zeta]=1,\;0\leq \zeta\leq \frac{1}{p},\\
  &\; \Eb[\zeta Z]= \frac{1}{p} H^{(-2)}_Z(p)\Big\}.
  \end{aligned}
\end{equation}
\item[\textup{(c)}] For any interval $J\subseteq(0,1],$ the multifunction $Y\to\mathbb{B}(Y,J)$ is convex-valued and has a closed graph in $\Lc_1(\Omega,{\mathcal F},P).$
\end{tightitemize}
\end{lemma}
\begin{proof}
The mapping $Z \mapsto  H^{(-2)}_Z(p)$ is identical to $Z\mapsto -p\avar^-_p(Z)$, where  $\avar^-_p(Z)$ is the Average Value-at-Risk representing profits (see \cite[Remark 2.20]{DDARbook}). Hence, the concavity and the positive homogeneity in statement (a) follow from the coherence of $\avar^-_p(\Cdot)$. Using the formula for the subdifferential of $\avar^-_p(\Cdot)$, we obtain statement (b).

Further, the continuity of $Z\mapsto \avar^-_p(Z)$ implies the continuity of the mapping $Z \mapsto  H^{(-2)}_Z(p)$ in $\Lc_1(\Omega,{\mathcal F},P)$, which in turn implies that the graph of the multifunction $Y\to\mathbb{B}(Y,J)$ is closed.
If $X,Z\in \mathbb{B}(Y,J)$, then the concavity of $Z \mapsto  H^{(-2)}_Z(p)$ implies that $tX+(1-t)Z\in \mathbb{B}(Y,J)$ entailing the convexity of the set
$\mathbb{B}(Y,J).$
\end{proof}

\section{Optimization with Stochastic Dominance Constraints}

Optimization problems with stochastic dominance constraints were introduced in \cite{ddarsiam} for integer orders and further analyzed in \cite{ddarmp};  a generalization to fractional orders is given in \cite{DDARbook}. Optimality conditions in Lagrangian form as well as in subdifferential form  show the role of utility functions as Lagrange multipliers.
   Our objective here is to derive new optimality conditions in subdifferential form for the inverse formulation of stochastic dominance. Optimality conditions in Lagrangian form
   established in \cite{DDARbook,dentcheva2008duality,Dentcheva2006InverseSD} show relations to distortion functionals and coherent measures of risk.

We analyze the following optimization problem:
\begin{align}
\min \ & \varphi(x)\label{dc-p1a}\\
\text{subject to}\ &
H^{(-2)}_{G(x)}(p) \ge H^{(-2)}_Y(p)\quad \forall p\in J=[\alpha,\beta]\subset(0,1], \label{dc-p2a}\\
& x\in \Yc. \label{dc-p3a}
\end{align}
The formulation fits the setting of \eqref{general-risk-opt} with
$\rho_p (G(x)) = H^{(-2)}_Y(p) - H^{(-2)}_{G(x)}(p).$

The function $\varphi(\Cdot)$ may be the composition \eqref{phi-def}; below we just use its general form for simplicity. In this section, we focus on the constraint operators and adopt the following general assumptions.
\begin{assumption}{~~}
\label{a:A-convex}
\begin{tightitemize}
  \item[\textup{(i)}] The function $\varphi(\Cdot)$ is convex and continuous.
  \item[\textup{(ii)}] The operator $G:\Xf\to\Lc_p(\Omega,\Fc,P)$ is norm-to-norm continuous and has the structure
$[G(x)](\omega) = g(x,\omega),
$
where $g:\Xf\times\Omega\to\Rb$ is a Carath\'eodory function and $g(\Cdot,\omega)$ is concave for all $\omega\in\Omega$.
  \item[\textup{(iii)}] The set $\Yc$ is a closed and convex subset of a Banach space $\Xf.$
\end{tightitemize}
\end{assumption}
The following requirement is our constraint qualification condition.
\begin{definition}
\label{uniformSDC}
Problem \eqref{dc-p1a}--\eqref{dc-p3a} satisfies the \emph{uniform inverse dominance condition} if
a point $\tilde{x} \in \Yc$ exists such that
\[
\inf_{p \in J}\big[ H^{(-2)}_{G(x)}(p) - H^{(-2)}_Y(p) \big] > 0.
\]
\end{definition}

The following result is known;  see \cite[Theorem~5.51]{DDARbook}.
\begin{theorem}
\label{t:KTspectral}
Suppose Assumption~\tref{a:A-convex} and the uniform inverse dominance condition are satisfied.
If $\hat{x}$ is an optimal solution of \eqref{dc-p1a}--\eqref{dc-p3a} 
then a spectral risk measure $\hat{\varrho}(\Cdot)$ and a constant $\kappa \ge 0$ exist such that $\hat{x}$
is also an optimal solution of the problem
\begin{gather}
\min_{x\in \Yc} \big\{  f(x) + \kappa \hat{\varrho}\big(G(x)\big) \big\}, \label{KTsrisk}\\
\kappa \hat{\varrho}\big(G(\hat{x})\big) = \kappa \hat{\varrho}(Y). \label{complrisk}
\end{gather}
Conversely, if for some  spectral  risk measure $\hat{\varrho}(\Cdot)$ and some $\kappa\ge 0$
an optimal solution $\hat{x}$ of \eqref{KTsrisk} satisfies \eqref{dc-p2a} with $[\alpha,\beta]=[0,1]$
and \eqref{complrisk}, then
$\hat{x}$ is an optimal solution of \eqref{dc-p1a}--\eqref{dc-p3a} with $[\alpha,\beta]=[0,1]$.
\end{theorem}
Recall that the normal cone to the convex set $\Yc\subset\Xf$ at a point $x\in\Yc$ is given by the formula:
\[
\Nc_{\Yc} (x) = \{d\in\Xf^*:\; \langle d, y-x\rangle \leq 0\quad\forall y\in\Yc\}
\]
with $\Xf^*$ being the topological dual to $\Xf.$

For a fixed $\omega,$ the concave subdifferential $\partial g(\hat{x},\omega)$ at $\hat{x}$
is as follows
\[
\partial g(\hat{x},\omega) = \{s\in \Xf^*:\,  g(x,\omega) \leq g(\hat{x},\omega) + \langle s(\omega),x-\hat{x}\rangle \  \forall x\in\Xf\}.
\]
In the theorem below, we use Definition \ref{d:F-subgradient} of a subgradient of a lattice-valued convex operator adapted to the concave case discussed here.
\begin{theorem}
\label{optimality_conditions_subdiff}
Suppose Assumption~\tref{a:A-convex} and the uniform dominance condition of Definition~\tref{uniformSDC} are satisfied for problem \eqref{dc-p1a}--\eqref{dc-p3a}.
If $\hat{x}$ is an optimal solution to that problem, then
a subgradient $S\in\partial [G(\hat{x})]$, a probability measure $\hat{\mu}$ supported on $[\alpha,\beta]$, and a constant $\kappa \ge 0$ exist such that the following inclusion
\begin{equation}
0\in \partial \varphi(\hat{x}) +\kappa S^* \int_\alpha^\beta \partial \avar^-_p (G(\hat{x})) \;\hat{\mu}(\D p)  +  {\Nc}_{\mathcal{X}}(\hat{x}), \label{sd-inv-subgradient}
\end{equation}
along with the complementarity condition \eqref{complrisk} are satisfied.
Conversely, if for some  probability measure $\hat{\mu}(\Cdot)$ on $[0,1]$ and some $\kappa\ge 0$ the inclusion \eqref{sd-inv-subgradient} is satisfied with a subgradient $S\in\partial [G(\hat{x})]$, as well as \eqref{complrisk} and \eqref{dc-p2a} hold with $[\alpha,\beta]=[0,1]$, then
$\hat{x}$ is an optimal solution of \eqref{dc-p1a}--\eqref{dc-p3a} with $[\alpha,\beta]=[0,1]$.
\end{theorem}
\begin{proof}
In view of Theorem~\ref{t:KTspectral}, we only need show the equivalence of conditions \eqref{sd-inv-subgradient}
and \eqref{KTsrisk}. Notice that if $\kappa = 0$, then condition \eqref{KTsrisk} means that $\hat{x}$ minimizes $f(\cdot)$ over the set $\Yc$ and constraint \eqref{dc-p2a} is not active. In that case, \eqref{KTsrisk} is equivalent to \eqref{sd-inv-subgradient} trivially.

Assuming that $\kappa>0$, the spectral risk measure in condition \eqref{KTsrisk} has the form
\[
\hat{\varrho}(Z) = \int_\alpha^\beta \avar^-_p (Z)\;\hat{\mu}(\D p)
\]
with some probability measure $\hat{\mu}$ supported on the interval $[\alpha,\beta].$

Theorem \ref{t:F-subgradient-structure} implies that the operator $G:\Xf\to\Lc_p(\Omega,\Fc,P)$ is subdifferentiable at $\hat{x}$ and its subdifferential is given by
\begin{multline*}
\partial [G(\hat{x})] = \big\{ S:\Xf\to\Lc_p(\Omega,\Fc,P) : [Sd] (\omega) = \langle s(\omega),d\rangle \\
\text{ for some } s\in\Lc_p(\Omega,\Fc,P;\Xf^*);\; s(\omega)\in\partial g(\hat{x},\omega)\big\}.
\end{multline*}
Under the assumptions of the theorem, the function
\[
\Lf(x) = f(x) + \kappa \hat{\varrho}\big(G(x)\big) = f(x) + \kappa \int_\alpha^\beta \avar^-_p (G(x))\;\hat{\mu}(\D p)
\]
is convex.
Hence $\hat{x}\in\Yc$ minimizes $\Lf (x)$ if and only if for some $S\in\partial G(\hat{x})$
\[
0\in\partial f(\hat{x}) + \kappa \partial_x  \hat{\varrho}\big(G(x)\big) = \partial f(\hat{x}) + \kappa \int_\alpha^\beta S^* \partial \avar^-_p (G(\hat{x})) \;\hat{\mu}(\D p).
\]
The right-hand side is equivalent to the one in \eqref{sd-inv-subgradient}.
\end{proof}

We can obtain another equivalent form of the integral term in condition \eqref{sd-inv-subgradient} for the non-trivial case of $\kappa>0$:
\begin{multline*}
\kappa\int_\alpha^\beta S^* \partial \avar^-_p (G(\hat{x})) \;\hat{\mu}(\D p) \\
 = \kappa\Big\{ S^*\int_\alpha^\beta \zeta(\hat{x},p,\omega)\;\hat{\mu}(\D p):~  \zeta(\hat{x},p,\omega)\in \partial \avar^-_p (G(\hat{x}))\Big\}\\
= \Big\{ S^*\int_\alpha^\beta \xi(\hat{x},p,\omega)\;\hat{\nu}(\D p):~  \xi(\hat{x},p,\omega)\in \partial H_{G(\hat{x})}^{(-2)} (p)\Big\}.
\end{multline*}
In the last formula $\hat{\nu}(\D p) = \frac{\kappa}{p} \hat{\mu}(\D p)$.

\section*{Conclusions}

We summarize our main contributions as follows. We provide new results on subdifferentiability of convex operators and their compositions which go beyond the theory of normal integrands.
The key technical advance is  the description of the subdifferential of a local convex operator between $\Lc_p$ spaces in Theorem \ref{t:F-subgradient-structure-integrand}. We offer a new proof technique using the Radon-Nikodym theorem for vector measures.
The key statement in Theorem \ref{t:F-subgradient-structure-integrand} allows for dealing with the new model with partial information in section \ref{sub-convex-composition}.
Additionally, we offer new optimality conditions in subdifferential form for optimization problems with inverse stochastic dominance constraints.

\bibliographystyle{abbrv}
\newcommand{\noopsort}[1]{} \newcommand{\printfirst}[2]{#1}
  \newcommand{\singleletter}[1]{#1} \newcommand{\switchargs}[2]{#2#1}

\end{document}